\documentclass{article}
\usepackage{soul}
\usepackage{graphicx}
\usepackage{amssymb,amsmath,amsfonts}
\usepackage{epsf}
\usepackage{wrapfig}

\usepackage{amsthm}
\newtheorem{thm}{Theorem}
\newtheorem{cor}[thm]{Corollary}
\newtheorem{lem}[thm]{Lemma}

\newtheorem{rem}[thm]{Remark}

\author{Peter J. Thomas\\ \small Case Western Reserve University
\\ \small Department of Mathematics \\ \small 10900 Euclid Avenue \\ \small 
Cleveland, Ohio 44106} % insert the authors here for use in running head
\usepackage{graphicx}
\newcommand{\E}{\mathbb{E}}
\newcommand{\V}{\mathbb{V}}

\begin{document}

\title{A Lower Bound for the First Passage Time Density of the Suprathreshold Ornstein-Uhlenbeck Process}
 % insert title - use \\ if it requires more than one line.

%\authorone[Case Western Reserve University]{Peter J. Thomas} % Affiliation is just the name of your university or institution

%\addressone{Department of Mathematics, 10900 Euclid Avenue, Cleveland Ohio, 44106} % Your postal address goes here.
\maketitle
\begin{abstract}
We prove that the first passage time density $\rho(t)$ for an Ornstein-Uhlenbeck process $X(t)$ obeying $dX=-\beta X\,dt + \sigma\,dW$ to reach a fixed threshold $\theta$ from a suprathreshold initial condition $x_0>\theta>0$ has a lower bound of the form $\rho(t)>k \exp\left[-p e^{6\beta t}\right]$ for positive constants $k$ and  $p$ for times $t$ exceeding some positive value $u$.  We obtain explicit expressions for $k, p$ and $u$ in terms of $\beta$, $\sigma$, $x_0$ and $\theta$, and discuss application of the results to the synchronization of periodically forced stochastic leaky integrate-and-fire model neurons.
\end{abstract}

\section{Introduction}
The distribution of the first passage time $\tau$ of an Ornstein-Uhlenbeck process to a fixed boundary is of interest in a variety of fields including mathematical neuroscience, where it describes the inter-spike interval distribution of a leaky integrate and fire (LIF) model neuron  driven by a combination of steady and fluctuating current sources \cite{CapocelliRicciardi1971Kybernetik}.   Unlike the density of the first passage time to a fixed boundary for Brownian motion, the FPT density for the OU process is not known in closed form.  Explicit expressions are known for the moments of the first passage time; in a neuroscience context the mean and variance are related respectively to the rate and variability of neuronal discharge \cite{RicciardiSato1988JApplProb,Wan+Tuckwell:1982:JTheorNeurobiol}.  Conditions guaranteeing the existence and smoothness of the FPT density to constant and moving boundaries have been studied \cite{Lehmann2002AdvApplProb,Pauwels1987JApplProb}.

Intuitively it is perhaps ``obvious" that the FPT density for an OU process is never identically zero, because
no matter how removed a time $t$ is from the typical time of first passage, there must exist a set of trajectories of positive measure that make a sufficiently large excursion from the mean behavior that arrival at the boundary is delayed until $t$ or later.  It is the purpose of this article to augment this intuitive argument by providing an explicit positive lower bound valid for the tail of the distribution.  
%at least in certain parameter ranges.  This result has been assumed implicitly in a variety of contexts (\textit{e.g.}~\cite{Tateno91?} page---) but no proof has been devised.  In addition to being of some intrinsic interest, obtaining a strictly positive lower bound for the FPT density of the OU process would bear on some open questions related to the existence of invariant densities for model systems arising from the neurosciences, particularly the periodically forced noisy integrate and fire neuron.  
We restrict attention to the lower bound problem for the ``suprathreshold" case, \textit{i.e.}~the case in which the initial value $x_0$ and the asymptotic mean $x_{\infty}$ of the process are on opposite sides of the threshold value. 
In the neural context this case corresponds heuristically to that of an LIF model neuron driven by a sufficiently strong injected current that it will reach its firing threshold in finite time, \textit{i.e.}~firing is not dependent on the presence of fluctuations. This case is therefore fundamentally distinct from that of ``stochastic resonance" phenomena, which occur in a subthreshold setting.  For the subthreshold case the initial and asymptotic value of the OU process lie on the same side of the capture boundary.  In this case asymptotic results have been proven that characterize the tail of the distribution at large times  \cite{GiornoNobileRicciardi1990AdvAppPr,NobileRicciardiSacerdote1985JApplPr}.  In contrast, the author is aware of no results providing strictly positive lower bounds for the FPT density in the suprathreshold case.  Such bounds would be of interest in the study of synchronization of neural activity \textit{via} periodic stimulation, for reasons to be discussed below.  
%work has been published on the \emph{lower} bound of the density, that is, bounds separating the density from zero.  
%
% 
% After defining the process and the density of interest, we review previous results in further detail in Section \ref{sec:prior}.  
% Section \ref{sec:motivation} outlines further the motivation of the problem connected with extension of the work by Wan and Tuckwell \cite{} and Keener, Hoppensteadt and Rinzel \cite{} among others \cite{Bressloff and Coombes, Ricciardi, Tateno} on firing time distributions and synchronization of integrate and fire model neurons.  
 %We then consider some approaches that contribute to developing intuition for the problem in Section \ref{sec:other} including construction of a novel traveling wave solution of the corresponding Fokker-Planck equation.  
Our main results are summarized as Theorem \ref{thm:main} and the following Remark.

\begin{thm}{Positive Lower Bound Theorem.}\label{thm:main}
Let $\{X(t),t\ge 0\}$ be an Ornstein-Uhlenbeck process satisfying the \^{I}to stochastic differential equation 
\begin{equation}
dX=-\beta X\,dt + \sigma\,dW \label{eq:dX-OUP}
\end{equation}
with initial condition
\begin{equation}
X(0)=x_0, \label{eq:X0}
\end{equation}
where $\beta$, $\sigma$ and $x_0$ are positive constants and $W(t)$ is a standard Wiener process.
Let $\theta$ be a fixed threshold satisfying $x_0>\theta>0$. 
Then there are positive constants $k, p$ and $u$ such that the first passage time density of the process to the threshold, $\rho_{X}(t)$, satisfies
\begin{equation}\label{ineq:main-result}
\rho_{X}(t)>k\exp\left[-p e^{6\beta t} \right]
\end{equation}
provided $t>u$.  
\end{thm}
\begin{rem}\label{rem:supplement-to-main-theorem}
Given particular values of the constants $\beta, \sigma, x_0$ and $\theta$, the inequality (\ref{ineq:main-result}) is satisfied for the following values of $k, p$ and $u$:
\begin{eqnarray}
k&=&\frac{1024\beta}{9\pi}\left(\frac{x_0}{\theta}-1\right)\\
p&=&1+\frac{\beta}{32}\left(\frac{\theta}{\sigma}\right)^2\\
u&=&\frac{1}{2\beta}\log\left[1+\left\{8\vee\left(1+\frac{x_0^2}{\theta^2}\right)\vee
\left(\frac{8\sigma^2}{\beta\theta^2}\right) \right\}\right]
\end{eqnarray}
where $\vee$ denotes the maximum operator.
\end{rem}

The proof is obtained from a geometric construction, elaborated in Sections \ref{sec:doob}-\ref{sec:OUP}, that exploits the well known change of variables by which an OU process may be written in terms of a standard Brownian Motion \cite{taylorkarlin}.  Sato \cite{Sato1977JApplProb} used the same transformation to obtain analytic results on the FPT distribution of the OU process to a constant boundary, in the subthreshold case.  While our results use the same change of variables, one cannot obtain the explicit lower bound sought directly from Sato's asymptotic results.  

The plan of the paper is as follows: in Section \ref{sec:doob} we introduce the needed change of coordinates and establish a lower bound for the first passage time of a standard Brownian motion to a piecewise linear approximation of a square root boundary.  In Section \ref{sec:OUP} we use these results to establish Theorem \ref{thm:main}.  In Section \ref{sec:discuss} we discuss potential applications of our result to synchronization of periodically forced leaky integrate and fire model neurons in the presence of additive noise, as well as related asymptotic results in the existing literature.

\section{Bounding the FPT distribution to a square root boundary for Brownian motion} \label{sec:doob}
\subsection{Coordinate Transformation}
Let $X(t)$ be an OU process described by Equations (\ref{eq:dX-OUP}-\ref{eq:X0}), and let $B(t)$ denote a standard Brownian motion, \textit{i.e.}~$B(t)$ satisfies 

%Let $B(t)$ be a standard Brownian motion, \textit{i.e.}~(if $W$ is a standard Wiener process, which is the same thing)
\begin{eqnarray}
dB&=&dW \label{eq:dB}\\
B(0)&\equiv& 0,\label{eq:B0}
\end{eqnarray}
%B(t)&\sim&\mathcal{N}(0,t)
%\end{eqnarray}
from which $\E[B]= 0$ and $\V[B]=t$.  
We will identify the sample space $\Omega$ for Brownian trajectories with $C[0,\infty)$.

%Let $X(t)$ be an Ornstein-Uhlenbeck process with constants $\beta,\sigma>0$ and initial value $x_0\in\mathbb{R}$ satisfying
%\begin{eqnarray}%\label{eq:dX-OUP}\\
%dX&=&-\beta X\,dt + \sigma\,dW  \\ 
%X(0)&\equiv& x_0. %\label{eq:X0}
%\end{eqnarray}
For $t\ge 0$ we have the representation of $X$ in terms of $B$:
\begin{equation}\label{eq:OUP-Brownian-rep}
X(t)=x_0e^{-\beta t}+\frac{\sigma e^{-\beta t}}{\sqrt{2\beta}}B\left(e^{2\beta t}-1\right).
\end{equation}
%which can be verified, for instance, by direct differentiation using \^{I}to's chain rule.
%From this representation it is easy to obtain the mean and variance of $V$:
%\begin{eqnarray}
%\E[X](t)&=&x_0e^{-\beta t}\\
%\V[X](t)&=&\sigma^2\left(\frac{1-e^{-2\beta t}}{2\beta}\right)
%\end{eqnarray}
%(To relate this OUP to the integrate and fire voltage satisfying $dV=(-V+\alpha)\,dt+\epsilon\,dW$, take $X(t)=V(t)-\alpha$, $x_0=v_0-\alpha < 0$, $\sigma=\epsilon$, and $\beta=1$.)
%The OUP is normally distributed, so its distribution function can be written in terms of the standard normal integral $\Phi(z)$, defined by 
%\begin{equation}
%\Phi(z)=\frac{1}{\sqrt{2\pi}} \int_{-\infty}^z e^{-x^2/2}\,dx
%\end{equation}

The problem of finding the first passage time distribution for an OU process with a fixed boundary is equivalent to the problem of finding the FPT distribution for a standard Brownian motion to a moving boundary with square root time dependence. %, as illustrated by Lemma \ref{lem:Doob-Trans}.
Given the parameter $\beta$ for the OUP, Equation (\ref{eq:dX-OUP}), define an exponentially rescaled time 
\begin{equation}\label{eq:s-from-t}
s=e^{2\beta t}-1, t\ge 0.
\end{equation}
Equivalently, for $s\ge 0$, we have $t=\left(\frac{1}{2\beta}\right)\log(s+1)$.

\begin{lem}\label{lem:Doob-Trans}
Let $X$ and $B$ be as defined in Equations (\ref{eq:dB}-\ref{eq:X0}), with $x_0>\theta>0$, and let $s$ be as defined in Equation (\ref{eq:s-from-t}).  Fix a particular Wiener process trajectory $\omega\in\Omega$.
Then $X(t,\omega)>\theta$ for $0\le t < \tau$ if and only if 
%\begin{equation}
$$B(s,\omega)> \sqrt{\frac{2\beta}{\sigma^2}}\left(\theta \sqrt{s+1} - x_0\right)$$
%\end{equation}
for all $0\le s < e^{2\beta \tau}-1$.
\end{lem}
\begin{proof}[Proof of Lemma \ref{lem:Doob-Trans}.]  
From Equation (\ref{eq:OUP-Brownian-rep}),
\begin{eqnarray}
&X(t)> \theta \label{ineq:X>theta}
&\mbox{ for } 0\le t < \tau \\
\iff &   x_0e^{-\beta t}+\frac{\sigma e^{-\beta t}}{\sqrt{2\beta}}B\left(e^{2\beta t}-1\right) > \theta 
&\mbox{ for } 0\le t < \tau \\
%\iff & \frac{\sigma e^{-\beta t}}{\sqrt{2\beta}} B\left(e^{2\beta t}-1\right)>\theta-x_0 e^{-\beta t}
%&\mbox{ for } 0\le t < \tau \\
\iff&  B\left(e^{2\beta t}-1\right)>\sqrt{\frac{2\beta}{\sigma^2}} \left(\theta-x_0e^{-\beta t}\right)e^{\beta t}
&\mbox{ for } 0\le t < \tau \\
%\iff& B(s)>\sqrt{\frac{2\beta}{\sigma^2}}\left(\theta-x_0e^{-\beta t}\right)e^{\beta t}  &\mbox{ for } 0\le s < e^{2\beta\tau}-1\\
\iff&B(s)>\sqrt{\frac{2\beta}{\sigma^2}}\left(e^{\beta t}\theta-x_0\right) &\mbox{ for } 0\le s < e^{2\beta\tau}-1\\
\iff & B(s)>\sqrt{\frac{2\beta}{\sigma^2}}\left(\theta\sqrt{s+1}-x_0\right)&\mbox{ for } 0\le s < e^{2\beta\tau}-1.\label{ineq:B>b}
\end{eqnarray}
This concludes the proof of Lemma \ref{lem:Doob-Trans}.
\end{proof}
Note that when $s=t=0$, we have $B(0)\equiv 0>\sqrt{\frac{2\beta}{\sigma^2}}(\theta-x_0)$, consistent with the assumption that $x_0>\theta$.   

Let $t'$ be the time of first passage of $X(t)$ to $\theta$.  Let $s'$ be the time of first passage of $B(s)$ to the boundary 
\begin{equation}\label{eq:bdyfcn}
b(s)=\frac{\sqrt{2\beta}}{\sigma}\left(\theta \sqrt{s+1} - x_0\right).
\end{equation}
The following corollary to Lemma \ref{lem:Doob-Trans} is immediate.
\begin{cor}\label{cor:doob-density}
Let the processes $X$, $B$ be defined as in Lemma \ref{lem:Doob-Trans}.  Let $b$ be the boundary function given by Equation (\ref{eq:bdyfcn}) and let $s$ be the rescaled time given by Equation (\ref{eq:s-from-t}).  Then as $dt\to 0$ the first passage time densities of $X$ and $B$ satisfy 
$$\Pr(t'\in[t,t+dt))=\Pr(s'\in[s,s+ds))$$
with
$$ds = \left(\frac{ds}{dt}\right)\,dt = 2\beta e^{2\beta t}\,dt = 2\beta(s+1)\,dt.$$
\end{cor}
Consequently $\rho_X(t)$, the first passage time density for $X$, can be obtained from the first passage time density for $B$, $\rho_B(s)$, as
\begin{equation}
\rho_X(t)=2\beta(s+1)\rho_B(s) = 2\beta e^{2\beta t}\rho_B\left(e^{2\beta t}-1 \right).
\end{equation}   
An explicit positive lower bound for $\rho_B$ will then provide a positive lower bound for $\rho_X$. 

\subsection{Piecewise Linear Approximation for $b(s)$}
In order to obtain a positive lower bound for $\rho_B$ we will need the slope $b'(s)$ of the boundary function $b(s)$:
\begin{equation}
b'(s)=\left(\frac{\sqrt{2\beta}}{\sigma}\right)\frac{\theta}{2\sqrt{s+1}} = \frac{\theta}{\sigma}\sqrt{\frac{\beta}{2}}\frac{1}{\sqrt{s+1}}.
\end{equation}
In particular, we have $b'(0)=\theta \left(\sqrt{\beta/2}\right)/\sigma>0$ and for all $s\ge 0$, $b'(s)>0$.  

Given $s'>0$ and a small time $ds>0$ let $\Omega_0\subset \Omega$ be the set of all trajectories in the sample space $\Omega=C[0,\infty)$ with first passage time $s'$ to the boundary $b(s)$ in the small interval $s'\in[s,s+ds)$.  We will decompose the set $\Omega_0$ into two subsets and obtain a positive lower bound for the measure of one of them.  First we construct a piecewise linear approximation of the boundary function $b$ (see Figure \ref{fig:piecewiselinear}).

Let $L_1$ be the half line $\{(s,a_1+s b_1) |s\ge 0\}$, where 
\begin{eqnarray}\label{eq:a1}
a_1&=&b(0)=(\theta-x_0)\sqrt{2\beta}/\sigma<0\\
\label{eq:b1}
b_1&=&b'(0)=(\theta/2)\sqrt{2\beta}/\sigma>0.
\end{eqnarray}   
The line $L_1$ is tangent to the boundary function $b(s)$ at $s=0$.  The boundary function has negative second derivative for all $s>0$, so  we have $a_1+sb_1>b(s)$ for all $s>0$.  Consequently every continuous trajectory $B(s)$ that begins at the initial condition $B(0)\equiv 0$ and meets the boundary $b(s)$ at some time $s'>0$ must first meet the line $L_1$ at some time $s''<s'$.   

Let $s'>0$ be given.  We define a second boundary line $L_2$ as the horizontal line with height 
\begin{equation}\label{eq:a2}
a_2=b(s')=\frac{\sqrt{2\beta}}{\sigma}\left(\theta \sqrt{s'+1} - x_0\right)
\end{equation} Lines $L_1$ and $L_2$ intersect at the point 
\begin{equation}\label{eq:s*}
s_*=\frac{a_2-a_1}{b_1}=2\left(\sqrt{s'+1}-1\right).
\end{equation}    
Note $s_*<s'$ provided $s'>0$.  The duration $\Delta$ of the interval $I=[s_*,s']$ is 
\begin{equation}
\Delta=s'-s*=s'+2(1-\sqrt{s'+1}) \label{eq:Delta}
\end{equation}
and satisfies $\Delta+2\sqrt{\Delta}=s'$. 
Figure \ref{fig:piecewiselinear} illustrates the geometry of $b(s), L_1$ and $L_2$.  
We note for future reference that when $s'=8$, $s_*=\Delta=4$, and that $s*$ is a  strictly increasing function of $s'$ for $s'\ge 0$. 
\begin{figure}[htbp] %  figure placement: here, top, bottom, or page
   \centering
   \includegraphics[width=6in]{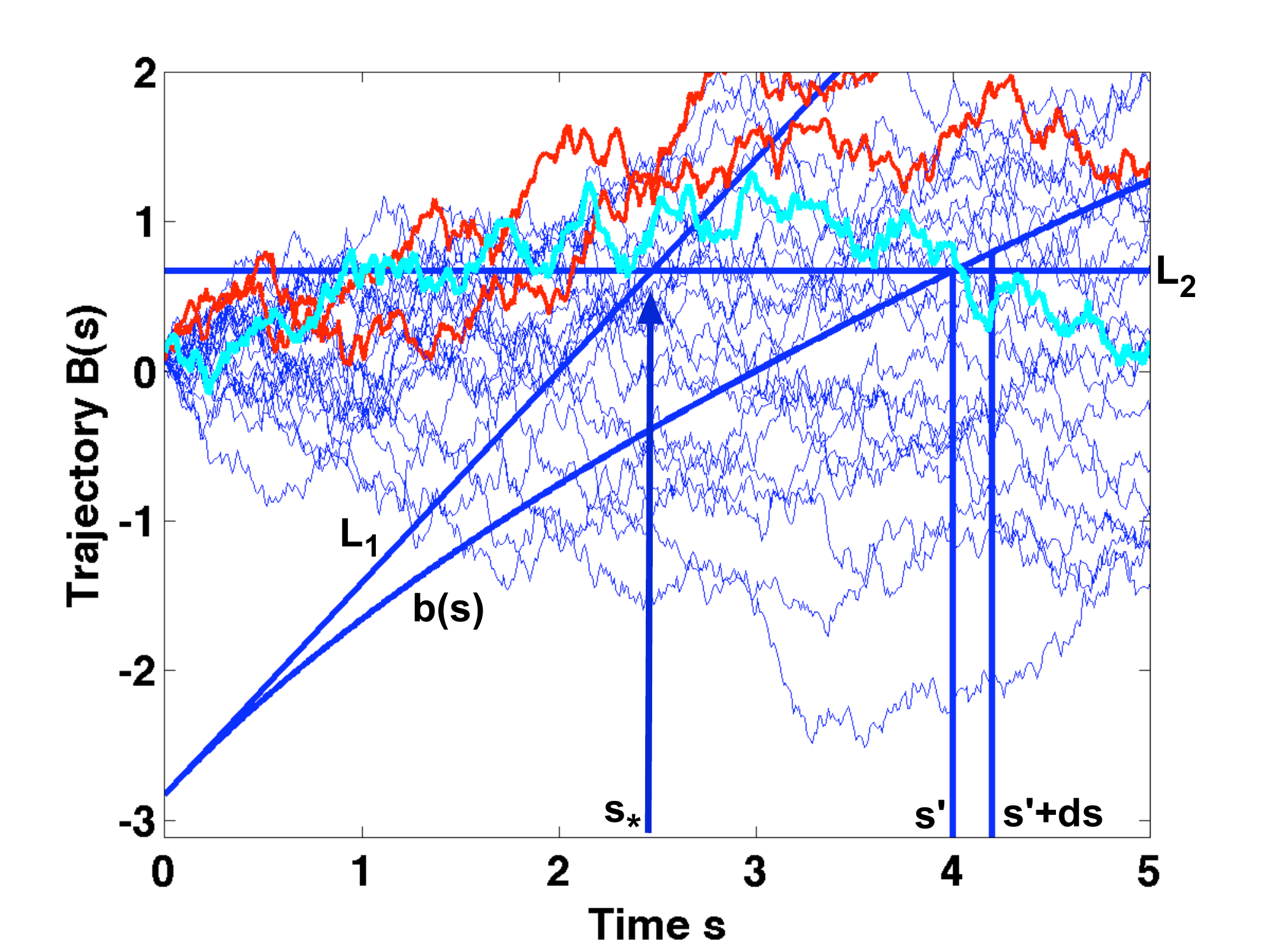} 
   \caption{Geometric construction of the piecewise linear boundary approximation, and sample trajectories.    Here the boundary $b(s)$ is the arc starting from $b(0)\approx -2.8$, and is given by Equation (\ref{eq:bdyfcn}) with $\sigma=.5, \beta=1, x_0=2$ and $\theta=1$.  The line $L_1$ begins at the same location and rises above the arc with slope $b_1=\sqrt{2}$.  Line $L_2$ is horizontal, intersecting line $L_1$ at $s_*=2\sqrt{5}-1\approx 2.47$ and the boundary $b(s)$ at $s'=4$.  
Approximate Brownian trajectories were generated using Matlab beginning at $B(0)\equiv 0$ and evolving with noise factor $\sigma=.5$. Of the twenty-five trajectories shown, three (thicker lines, shown in red and cyan) cross $L_1$ for the first time at $s''>s_*$. One of these trajectories (thickest line, shown in cyan) then crosses $L_2$ between in the interval $[s',s'+ds)$; here $ds=0.2$ for illustration.  Because $L_2$ lies below $b(s)$ for $s>s'$,  the cyan trajectory must also have its first passage to $b(s)$ within the interval $[s',s'+ds)$.  The cyan trajectory is a typical element of $\Omega_2$.}
   \label{fig:piecewiselinear}
\end{figure}

Given $ds>0$, we have defined $\Omega_0$ to be the collection of all Wiener process trajectories $\omega\in \Omega=C[0,\infty)$ such that the process $X(s,\omega)$ makes its first passage from $X(0)\equiv 0$ to $b(s)$ in the interval $[s',s'+ds)$. %\footnote{Note that $X(s,\omega)$ is a.s.~continuous, cite e.g. Karatzas and Shreve.}  
We further define $\Omega_1$ to be the subset of $\Omega_0$ consisting of those trajectories in $\Omega_0$ that make their first passage from $X(0)\equiv 0$ to the line $L_1$ at any time $s''\ge s_*$.  Figure \ref{fig:piecewiselinear} shows an example of such a trajectory highlighted in cyan (light blue).  Starting from $X(0)=0$ the cyan trajectory remains above the line $L_1$ until passing the intersection of $L_1$ and $L_2$ at $s=s_*$, after which it crosses $b(s)$ for the first time in the interval $[s',s'+ds)$.  Of the trajectories in $\Omega_1$, some -- like the trajectory shown in cyan -- also cross the line $L_2$ somewhere in the interval $[s',s'+ds)$.  Denote the set of trajectories in $\Omega_1$ with this property by $\Omega_2$.  

Clearly, $\Omega_2\subset\Omega_1\subset\Omega_0$.  Because $b(s)$ is continuous and strictly increasing, every trajectory that remains above line $L_1$ until at least $s=s_*$ and then intersects line $L_2$ for the first time in the interval $[s',s'+ds)$ must also make its first passage to $b(s)$ in the interval $[s',s'+ds)$.  
Hence if the density of trajectories in $\Omega_2$ crossing $L_2$ in $[s',s'+ds)$ is strictly positive, so is the first passage time density of all trajectories in $\Omega_0$,
which is the FPT density of interest. We next derive an explicit positive lower bound on the FPT density of trajectories in $\Omega_2$, hence also a lower bound for the FPT density of trajectories in $\Omega_0$.  

\subsection{Strictly positive lower bound for $\rho_B$}

The first passage time density for a standard Brownian motion from  $B(0)\equiv 0$ to a boundary moving with constant velocity is well known (\textit{e.g.}~\cite{cox+miller:1965}, p.~221; \cite{loader+deely:1987:JStatCompSim}). %, \cite{Redner2001}).
The density of first passage from $B(0)\equiv 0$ to $L_1$ at time $s$ is 
\begin{equation}
g_{01}(s)=\frac{|a_1|}{(2\pi s^3)^{1/2}}\exp\left(-\frac{(a_1+b_1s)^2}{2s}\right)
\end{equation}
where $a_1$ and $b_1$ are given by Equations (\ref{eq:a1}-\ref{eq:b1}).
Similarly, the density of first passage from a starting position $a_1+b_1s$ at time $s$ to 
the line $L_2$ at time $s'>s$ is given by 
\begin{equation}
g_{12}(s'|s)=\frac{|(a_1+b_1s)-a_2|}{(2\pi (s'-s)^3)^{1/2}}\exp\left(-\frac{((a_1+b_1s)-a_2)^2}{2(s'-s)}\right).
\end{equation}
Because of the almost sure continuity of Brownian trajectories,
we may write the first passage time density $g_2(s')$ for trajectories in $\Omega_2$ from initial position $B(0)\equiv 0$ to line $L_2$ at time $s'$ %in the interval $[s',s'+ds)$ 
as a convolution of $g_{01}$ with $g_{12}$.  Using $l_1(s)$ to denote $(a_1+b_1s)$, we write
\begin{eqnarray}
g_2(s') &=& \int_{s=s_*}^{s'}g_{01}(s)g_{12}(s'|s)\,ds\\
\label{eq:threefactors}
&=& \frac{|a_1|}{2\pi}\int_{s_*}^{s'}\frac{l_1(s)-a_2}{((s'-s)s)^{3/2}}
\exp\left(-\frac{l_1(s)^2}{2s}-\frac{(l_1(s)-a_2)^2}{2(s'-s)}\right)\,ds \label{eq:g2conv}
\end{eqnarray}
where the upper end point ($s'$) and lower end point ($s_*$) for the integration both depend on $s'$, as does the height of $L_2$, namely $a_2=b(s')$. 
The integrand in Equation (\ref{eq:g2conv}) approaches zero at both endpoints, but is stricly positive in the interior of the interval $I=\{s|s_*\le s\le s'\}$. To obtain a positive lower bound for the integral we rewrite the integrand as a product of three factors:
\begin{eqnarray}\label{eq:g2-def}
g_2(s')&=&\frac{|a_1|}{2\pi}\int_{s\in I} F_1(s,s')F_2(s,s')F_3(s,s')\,ds\\
F_1(s,s')&=&a_1-a_2(s')+b_1s\\
F_2(s,s')&=&(s(s'-s))^{-3/2}\\
F_3(s,s')&=&\exp\left(-\frac{(s'-s)l_1^2(s)+s(l_1(s)-a_2)^2}{2s(s'-s)}\right)
=\exp\left(\frac{-\beta}{4\sigma^2}\frac{Q_1(s,s')}{Q_2(s,s')}\right)
\end{eqnarray}
where we introduce the notation
\begin{eqnarray} \label{eq:Q1}
Q_1(s,s')&=&\left(s'l_1^2(s)-2sl_1(s)a_2+sa_2^2\right)\left(2\sigma^2/\beta\right)\\ %\nonumber
%&=&s^2(b_1^2s'-2a_2b_1)+s(a_2^2-2a_1a_2+2a_1b_1s')+a_1^2s'\\ \nonumber
&=& s^2\theta(4x_0 + \theta(s'-4\sqrt{1+s'})) +4s'(x_0-\theta)^2 
\\ \nonumber
&&-4s(x_0^2+(s'-2)x_0\theta+(-1-2s'+2\sqrt{1+s'})\theta^2)\\
Q_2(s,s')&=&s(s'-s).
\end{eqnarray}
By construction, $F_1(s_*,s')\equiv 0$ for all $s'>0$ (recalling that $s_*$ depends on $s'$) and $F_1(s,s')>0$ for $s>s_*$.  On the other hand,  $F_3(s,s')\to 0$ as $s\to s'$.  Thus the integrand approaches zero at both ends of the interval of integration.  In order to obtain a strictly positive lower bound on the integral we restrict integration to a subinterval of $[s_*,s']$.
Let $\eta,\nu$ be chosen in the interval $0<\eta<\nu<1$ and set
\begin{eqnarray}
s_1&=&s_* + \eta \Delta \\
s_2&=& s_* + \nu \Delta 
\end{eqnarray}
where as above $\Delta=(s'-s_*)$ is the duration of the integration interval.
Then $s_*<s_1<s_2<s'$.  Let $I'$ denote the interval $I'=[s_1,s_2]\subset I$. Each factor $F_1, F_2, F_3$ is nonnegative on the interval $I$.  Consequently we have the inequality
\begin{equation}
g_2(s')\ge \frac{|a_1|}{2\pi}\int_{s\in I'}F_1(s,s')F_2(s,s')F_3(s,s')\,ds.
\end{equation}
%%choose a subinterval on which to obtain lower estimates of each of the factors $F_1, F_2, F_3$ and $F_4$:
%\begin{eqnarray}
%F_1&=&a_1+b_1s-a_2\\
%F_2 &=& ((s'-s)s)^{-3/2}\\
%F_3 &=&\exp\left(-\frac{(a_1+b_1s)^2}{2s}\right)\\
%F_4 &=& \exp\left(-\frac{(a_1+b_1s-a_2)^2}{2(s'-s)}\right)
%\end{eqnarray}
%(The trick is to choose the interval well.  If $A=s'-s_*$, how about $s'+A/3<s<s'+A/3$?)
We consider each factor in turn.   

Factor $F_1(s,s')$ is strictly increasing in $s$ for all $s, s'$ so for all $s\in I'$ we have 
\begin{equation}
F_1(s,s')\ge F_1(s_1,s')=\eta \frac{\theta\sqrt{2\beta}}{2\sigma} \Delta.
\end{equation}
%where as above $\Delta=(s'-s_*)>0$.

The quadratic $Q_2(s,s')$, which also appears in factor $F_2(s,s')=Q_2(s,s')^{-3/2}$, has a global maximum (in $s$) of $(s')^2/4$ at $s=s'/2$.  In the range $s\in[s_1,s_2], Q_2(s,s')$ is strictly positive.  Consequently for all $s\in I'$ we have
\begin{equation}
F_2(s,s')\ge F_2(s'/2,s')=\frac{8}{(s')^3}.
\end{equation}

From Equation (\ref{eq:threefactors}) it is clear that both $Q_1(s,s')$ and $Q_2(s,s')$ are strictly positive for $s\in I'$.  The factors $\beta$ and $\sigma$ are also positive, therefore for $s\in I'$,
\begin{equation}
F_3(s,s')\ge\exp\left(-\frac{\beta}{4\sigma^2} \frac{\max_{I'}Q_1(s,s')}{\min_{I'}Q_2(s,s')} \right).
\end{equation}
As established in the next Lemma, $Q_1$ is just a quadratic function of $s$ with second derivative that is guaranteed to be positive, provided $s'>8$.  

\begin{lem}\label{lem:Q1concaveup}
If $s'>8$ and $x_0>\theta>0$ then the quadratic function $Q_1(s,s')$ defined by Equation (\ref{eq:Q1}) has positive second derivative with respect to $s$.
\end{lem}
\begin{proof}
The second derivative of $Q_1$ with respect to $s$ is  constant in $s$, 
$\partial^2 Q_1/\partial s^2=2\theta(4x_0+\theta(s'-4\sqrt{1+s'}))$.   Suppose $s'>8$.  Then $(s')^2>8s'$, so $(4+s')^2>16+16s'$.  Taking square roots, we obtain $4+s'>4\sqrt{1+s'}$, which means that $1>\sqrt{1+s'}-s'/4$, and since $\theta>0$ by assumption, $\theta>\theta(\sqrt{1+s'}-s'/4)$.  By the assumption on $x_0$, $x_0>\theta>\theta(\sqrt{1+s'}-s'/4)$.  Consequently 
$x_0+\theta(s'/4-\sqrt{1+s'})>0$, from which the conclusion follows.  
\end{proof}

Because $Q_1(s,s')$ has constant positive second derivative with respect to $s$, it achieves its maximum on $I'$ either at $s_1$ or at $s_2$.  
The quadratic $Q_2(s,s')$ has constant negative second derivative with respect to $s$ and achieves its minimum at one of the end points as well.    Therefore we may conclude that\footnote{For compactness  we henceforth write $Q_i(s)$ for $Q_i(s,s')$, leaving the dependence on $s'$ implicit.}
\begin{equation}
F_3(s)\ge\exp\left(-\frac{\beta}{4\sigma^2} \left[\frac{Q_1(s_1)\vee Q_1(s_2)}{Q_2(s_1)\wedge Q_2(s_2)}\right] \right)
\end{equation}
where $\vee$ denotes the greater and $\wedge$ the lesser of the two arguments.

Therefore in terms of the quadratics $Q_1$, $Q_2$, and the given constants, we have the following %strictly positive 
lower bound for the first passage time density at time $s'$ 
\begin{equation}
g_2(s')\ge \frac{|a_1|}{2\pi}(\nu-\eta)\Delta\left(\eta \frac{\theta\sqrt{2\beta}}{2\sigma} \Delta\right)\left(\frac{8}{(s')^3}\right)\exp\left(-\frac{\beta}{4\sigma^2} \left[\frac{Q_1(s_1)\vee Q_1(s_2)}{Q_2(s_1)\wedge Q_2(s_2)}\right] \right)
\end{equation}
provided $s'>8$.  

Note that for $s'>8$, $\Delta(s')%=s'-s_*=s'-2(\sqrt{s'+1}-1)
> \frac{2}{3}(s'-2)>4$. 
Since $|a_1|=(x_0-\theta)\sqrt{2\beta}/\sigma>0$ we may write
\begin{eqnarray}\nonumber
g_2(s') &>& (\nu-\eta)\frac{(x_0-\theta)\sqrt{2\beta}}{2\pi\sigma}\left(\frac{8}{9} \frac{\theta\sqrt{2\beta}}{\sigma} \eta \right)\left(\frac{8}{(s')^3}\right)\exp\left(-\frac{\beta}{4\sigma^2} \left[\frac{Q_1(s_1)\vee Q_1(s_2)}{Q_2(s_1)\wedge Q_2(s_2)}\right] \right)\\
&=&(\nu-\eta)\eta\frac{64\beta\theta}{9\pi\sigma^2}\left(\frac{x_0-\theta}{(s')^3}\right)\exp\left(-\frac{\beta}{4\sigma^2} \left[\frac{Q_1(s_1)\vee Q_1(s_2)}{Q_2(s_1)\wedge Q_2(s_2)}\right] \right)
\end{eqnarray}

For any number $\alpha$, $0< \eta \le \alpha\le \nu<1$, we may write a point $s\in I'$ as $s=s_*+\alpha(s'-s_*)$ and $(s'-s)=(1-\alpha)(s'-s_*)$. Therefore we can write $Q_2(s)$ as $Q_2(s)=(1-\eta)s_*(s'-s_*)+\eta(1-\eta)(s'-s_*)^2$.
When $s'=8$ we have $\Delta=4$ and $s_*=4$ from Equations (\ref{eq:s*}, \ref{eq:Delta}).  Both $\Delta$ and $s_*$ are strictly increasing as functions of $s'$, so for $s'>8$ we have $\Delta>4$ and $s_*>4$.  
Therefore for $s'>8$ we have 
\begin{equation}Q_2(s_1)\wedge Q_2(s_2)=16(1-\nu^2).\end{equation} % since $0<\eta \le \alpha \le \nu < 1$. 

It remains to obtain an estimate on the quadratic $Q_1$.
\begin{lem}\label{lem:Q1}
Assume $x_0>\theta>0$.  Fix the constants $0<\nu<\eta<1$ and $s_1=s_*+\eta\Delta, s_2=s_*+\nu\Delta$  with $\Delta=(s'-s_*)$, and $s_*$ as in Equation (\ref{eq:s*}); let $Q_1$ be given by Equation (\ref{eq:Q1}) and let $s'>8\vee(1+x_0^2/\theta^2)$.  Then 
$$Q_1(s_1)\vee Q_1(s_2)\le \Delta^2\theta^2 s'.$$
\end{lem}
\begin{proof}
For any number $0\le \alpha \le 1$, we have a corresponding point $s\in I$ for which $s=s_*+\alpha(s'-s_*)$ and $(s'-s)=(1-\alpha)(s'-s_*)$.
As shown in Lemma \ref{lem:Q1concaveup}, $Q_1$ has constant positive second derivative with respect to $s$ when $s'>8$.  Hence its largest value for $\alpha\in[0,1]$ exceeds its largest value for $\alpha\in [\nu,\eta]$. From Equation (\ref{eq:Q1}) we obtain
\begin{eqnarray}
Q_1(s)\beta/\left(2\sigma^2\right)&=&(s'-s)l_1^2(s)+s(l_1(s)-a_2)\\
&=&s'l_1^2-sl_1^2+sl_1^2-2sa_2l_1+sa_2^2.
\end{eqnarray}
Note that by construction $l_1(s)=b_1(s-s_*)+a_2$.  Therefore
\begin{eqnarray}
Q_1(s)\beta/\left(2\sigma^2\right)&=&(s'-s)(b_1(s-s_*)+a_2)^2+s(b_1(s-s_*))^2\\
&=&(s'-s)(b_1^2(s-s_*)^2+2b_1a_2(s-s_*)+a_2^2)+sb_1^2(s-s_*)^2.
\end{eqnarray}
Write $s'-s=(1-\alpha)\Delta$ and $s-s_*=\alpha\Delta$. Then
\begin{eqnarray}\nonumber 
\frac{Q_1(s(\alpha))\beta}{2\sigma^2}&=&(1-\alpha)\Delta(b_1^2\alpha^2\Delta^2+2b_1a_2\alpha\Delta + a_2^2) + (s'-(1-\alpha)\Delta)(\alpha^2\Delta^2)b_1^2\\ \nonumber 
&=&\Delta(2(1-\alpha)\alpha b_1a_2\Delta + (1-\alpha)a_2^2+s'b_1^2\alpha^2\Delta)\\ \nonumber 
&=&\Delta\frac{2\beta}{\sigma^2}((1-\alpha)\alpha\theta(\theta\sqrt{s'+1}-x_0)+(1-\alpha)(\theta\sqrt{s'+1}-x_0)^2\\ 
&&+\alpha^2\Delta s'\theta^2/4 )\\
Q_1(s_*)&=&Q_1\left(s(\alpha)|_{\alpha=0}\right)= 4\Delta(\theta\sqrt{s'+1}-x_0)^2   \\
Q_1(s')&=&Q_1\left(s(\alpha)|_{\alpha=1}\right)=\Delta^2\theta^2 s'.
\end{eqnarray}
For $s'>8$ we have $\Delta>4$, as shown previously.  Therefore for $s'>8$ we have the following inequalities
\begin{eqnarray}
\sqrt{\frac{s'}{s'+1}}&<&1<\frac{\sqrt{\Delta}}{2}\\
0&>&\frac{1}{2\sqrt{s'+1}} - \frac{\sqrt{\Delta}}{4\sqrt{s'}}\\
1&>&3-2\sqrt{2}>\sqrt{s'+1}-\frac{1}{2}\sqrt{\Delta s'},
\end{eqnarray}
since $\sqrt{s'+1}-\frac{1}{2}\sqrt{\Delta s'}=3-2\sqrt{2}$ when $s'=8$ and 
$\frac{d}{ds}\left(\sqrt{s+1}-\frac{1}{2}\sqrt{\Delta s}\right) = \frac{1}{2\sqrt{s+1}} - \frac{\sqrt{\Delta}}{4\sqrt{s}}.$
Therefore, 
\begin{eqnarray}
x_0>\theta&>&\theta\sqrt{s'+1}-\frac{\theta}{2}\sqrt{\Delta s'}\\
\sqrt{\Delta s'}\theta/2 &>& \theta\sqrt{s'+1}-x_0 \label{ineq:foo}\\
\Delta s'\theta^2/4 &>&(\theta\sqrt{s'+1}-x_0)^2 \label{ineq:bar}\\
Q_1(s')&>&Q_1(s_*).
\end{eqnarray}
Inequality (\ref{ineq:bar}) follows from (\ref{ineq:foo}) because $(\theta\sqrt{s'+1}-x_0)>0$, by the assumption $s'>\left(1+x_0^2/\theta^2\right)$.
Therefore $$Q_1(s_1)\vee Q_1(s_2)\le Q_1(s')\le \Delta^2\theta^2 s'$$
as was to be shown.
\end{proof}

We are now able to state Lemma \ref{lem:main} which leads directly to Theorem \ref{thm:main}.
\begin{lem}\label{lem:main}
Let $x_0>\theta>0$ and let $s'$ be sufficiently large that 
\begin{equation}
s'>8\vee\left(1+\frac{x_0^2}{\theta^2}\right)\vee\left(\frac{8\sigma^2}{\beta\theta^2}\right).
\end{equation}
Then for the standard Brownian motion given by Equations (\ref{eq:dB}-\ref{eq:B0}), the conditional first passage time density $g_2$ defined by Equation (\ref{eq:g2-def}) for $B(s)$ to arrive at the boundary $b(s)$ given by Equation (\ref{eq:bdyfcn}) satisfies
\begin{equation}
g_2(s')> \frac{512}{9\pi}\left(\frac{x_0}{\theta}-1\right)\left(s'\right)^{-6}
\exp\left[-\frac{\beta}{32}\left(\frac{\theta}{\sigma}\right)^2 \left(s'\right)^3 \right].
\end{equation}
\end{lem}
\begin{proof}
Lemmas \ref{lem:Q1concaveup} and \ref{lem:Q1} and the subsequent remarks establish that when $s'>8\vee\left(1+x_0^2/\theta^2\right)$, the density $g_2$ defined by Equation (\ref{eq:g2-def}) satisfies
\begin{equation}\label{eq:g2-temp}
g_2(s')\ge A\eta(\nu-\eta)\exp\left[-2 B/(1-\nu^2)\right]
\end{equation}
where the parameters $\eta$ and $\nu$ may be chosen arbitrarily subject to the constraints $0<\eta<\nu<1$, and the positive terms $A$ and $B$ (which depend on $s'$) are given by
\begin{eqnarray}
A(s')&=&\frac{64\beta\theta(x_0-\theta)}{9\pi\sigma^2\left(s'\right)^3}\\
B(s')&=&\frac{\beta\theta^2\Delta^2 s'}{128\sigma^2}.
\end{eqnarray}
Let us rewrite the right hand side of Equation (\ref{eq:g2-temp}) as $g_2=AM(\eta,\nu)$, introducing 
\begin{equation}
M(\eta,\nu)=\eta(\nu-\eta)\exp\left[-2 B/(1-\nu^2)\right].
\label{eq:M-eta-nu}
\end{equation}
Because the choice of $\eta$ and $\nu$ is arbitrary, within the constraints, $g_2$ must be bounded below by the supremum of $M$ over the region $0<\eta<\nu<1$.  Clearly $M$ is positive and differentiable within this region, and $M\to 0$ as $(\eta,\nu)$ approach any of the boundaries $\eta=0$, $\nu=1$, $\eta=\nu$. So the supremum will occur at a point $(\eta_+,\nu_+)$ in the interior of the constraint region.  Setting $\partial M/\partial \eta $ to zero to find the critical point gives $\nu_+=2\eta_+$.  Substituting back into (\ref{eq:M-eta-nu}) and differentiating gives $\nu_+=\sqrt{B+1-\sqrt{B^2+B}}$.  Substituting again to find the maximum value yields
% for details see notes from 2011/1/6 %
\begin{equation}\label{eq:M-max}
M_{\mbox{max}}=\frac{1-C}{4}\exp\left[\frac{-2B}{C} \right]
\end{equation}
where we introduce
\begin{equation}
C=\sqrt{B^2+2B}-B.
\end{equation}
Note that for all $B>0$, we have $0<C<1$.  The value of $C$ increases monotonically and $C\to 1$ from below as $B\to\infty$.  Since $B$ is directly proportional to $s'$, we can find a value of $s'$ sufficiently large that $C$ exceeds any value less than 1.  For $B>1/4$, for example, we have $C>1/2$. Consequently we can bound below the exponential factor in (\ref{eq:M-max}): 
\begin{equation}\label{ineq:exponential-for-M}
\exp\left[\frac{-2B}{C} \right]\ge\exp\left[-4B\right]
\end{equation}
whenever $B>1/4$.  

The factor $(1-C)\to 0$ as $B\to\infty$. But provided $B\ge 1$ we have the following: 
\begin{eqnarray}
8B^2-8B+1&>&0\\
(4B^2+4B-1)^2&>&16B^2(B^2+2B) \label{ineq:foo2} \\
4B^2+4B-1 &>&4B\sqrt{B^2+2B} \label{ineq:bar2} \\
1-C=1+B-\sqrt{B^2+2B}&>&\frac{1}{4B}.\label{ineq:1-minus-C}
\end{eqnarray}
Inequality (\ref{ineq:bar2}) follows from (\ref{ineq:foo2}) because we clearly have $4B^2+4B-1>0$ provided $B\ge 1$.

Putting together (\ref{eq:M-max}), (\ref{ineq:exponential-for-M}) and (\ref{ineq:1-minus-C}), we may write
$M_{\mbox{max}}>\exp\left[-4B \right]/(16 B)$
provided $B>1$.  But requiring $B>1$ is equivalent to $s'>128\sigma^2/\left(\beta\Delta^2\theta^2\right)$.
As noted earlier, when $s'>8$ we have $\Delta>4$, so it is enough to require that 
$s'>8\sigma^2/\left(\beta\theta^2\right)$.
Therefore under the hypotheses of the lemma, $g_2(s')>A(s')\exp\left[-4B \right]/(16 B)$, from which it follows that 
$$
g_2(s')> \frac{512}{9\pi\Delta^2}\left(\frac{x_0}{\theta}-1\right)\left(s'\right)^{-4}
\exp\left[-\frac{\beta}{32}\left(\frac{\theta\Delta}{\sigma}\right)^2 s' \right].
$$
Since $\Delta(s') < s'$ for $s'>0$, and since $\beta>0$, we have 
$$
g_2(s')> \frac{512}{9\pi}\left(\frac{x_0}{\theta}-1\right)\left(s'\right)^{-6}
\exp\left[-\frac{\beta}{32}\left(\frac{\theta}{\sigma}\right)^2 \left(s'\right)^3 \right]
$$
as required.  This completes the proof of Lemma \ref{lem:main}.
\end{proof}

\section{Strictly positive lower bound for $\rho_X$}
\label{sec:OUP}

The proof of our main theorem follows from Lemma \ref{lem:main} and Corollary \ref{cor:doob-density}.  
\begin{proof}[Proof of Theorem \ref{thm:main} and Remark \ref{rem:supplement-to-main-theorem}]
Let $\{X(t),t\ge 0\}$ be an Ornstein-Uhlenbeck process satisfying Equations (\ref{eq:dX-OUP}-\ref{eq:X0}) and let $\beta$, $\sigma$ and $x_0$ be positive constants; let $\theta$ be a constant satisfying $x_0>\theta>0$.
To simplify notation we introduce the positive constants
\begin{eqnarray}
K&=&\frac{512}{9\pi}\left(\frac{x_0}{\theta}-1\right)\\
H&=&\frac{\beta}{32}\left(\frac{\theta}{\sigma}\right)^2.
\end{eqnarray}
Let $u=\frac{1}{2\beta}\log\left[1+\left\{8\vee\left(1+\frac{x_0^2}{\theta^2}\right)\vee
\left(\frac{8\sigma^2}{\beta\theta^2}\right) \right\}\right]$, as in Remark \ref{rem:supplement-to-main-theorem}.
Then by virtue of the time rescaling, \textit{cf.}~Equation (\ref{eq:s-from-t}), from 
Corollary \ref{cor:doob-density} and Lemma \ref{lem:main}, provided $t>u$ we may write
\begin{eqnarray}
\rho_X(t)&\ge&2\beta e^{2\beta t}g_2\left(e^{2\beta t}-1\right)\\
&>&2\beta e^{2\beta t} K \left(e^{2\beta t}-1\right)^{-6}\exp\left[-H\left(e^{2\beta t}-1\right)^3 \right]\\
&>&2\beta K e^{-10\beta t}\exp\left[-H e^{6\beta t} \right].
\end{eqnarray}
Let $p=1+H$ and $k=2\beta K$.  By hypothesis $s=\left(e^{2\beta t}-1\right)>8$ so $\beta t > \log 3 > 1$. 
It follows that $e^{-10\beta t}>\exp\left[-e^{6\beta t} \right]$.
Therefore $\rho_X(t)> 2\beta K \exp\left[-(1+H)e^{6\beta t}\right]=k\exp\left[-pe^{6\beta t} \right]$.  This concludes the proof of both Theorem \ref{thm:main} and Remark \ref{rem:supplement-to-main-theorem}.
\end{proof}

The form of the constants and the bound are consistent with the following \emph{heuristic} relationships between the parameters and the tail of the decay.  If $\beta$ is small, the initial condition is far from the threshold
and the noise is large ($x_0,\sigma\gg\theta$), it is reasonable that the tail of the FPT density should decay relatively slowly.  If $\beta$ is large, the initial condition is close to the threshold and the noise is small, it is reasonable to expect a faster decay of the density at long times.

\section{Discussion}
\label{sec:discuss}

\subsection{Relation to other asymptotic FPT density results}
\label{sec:discuss1}

Nobile \textit{et al.}~\cite{NobileRicciardiSacerdote1985JApplPr} derived an asymptotic expression for the first passage time distribution $\rho_{\theta}(t|x_0)$ for an Ornstein-Uhlenbeck process with limiting mean value $x_{\infty}$ to go from initial value $x_0$ to a boundary of height $\theta$ in the limit in which $\theta\gg x_0,x_{\infty}$.  They prove that 
\begin{equation}
\rho_{\theta}(t|x_0)=\frac{1}{\tau}e^{-t/\tau} + o\left(\frac{1}{\tau}e^{-t/\tau}\right), \mbox{ as } \theta\to\infty,
\end{equation}
where $\tau$ is the mean first passage time.   These same authors extended this result to a  broader class of diffusion processes in \cite{NobileRicciardiSacerdote:1985:JApplPrB}.  Similarly, Giorno \textit{et al.}~showed in \cite{GiornoNobileRicciardi1990AdvAppPr} that the FPT distribution showed good agreement with the form 
$$\rho(t)\approx Z(t)e^{-\lambda t}$$
for passage to a periodically varying boundary $\theta(t)$, in cases where $\theta(t)\gg x_0=x_{\infty}$.
The limit $\theta\gg x_0,x_{\infty}$ corresponds to the \emph{subthreshold} case, in contrast to the \emph{suprathreshold} case $x_0>\theta>x_{\infty}$ which concerns us here. 

%Early paper on application of OU process FPT to integrate-and-fire spike time distributions: \cite{Ricciardi+Sacerdote:1979:BiolCyb}.  This paper derived an expression for the Laplace transform of the full first passage time distribution for the OU process as a sum of parabolic cylinder functions.  Earlier paper on diffusion approximation: Ricciardi and QQQ.

Pauwels \cite{Pauwels1987JApplProb} gave conditions on the noise ($\sigma$) and drift ($b$) parameters guaranteeing that the stochastic differential equation $\{dX_t=\sigma(X_t)\,dB_t+b(X_t)\,dt,\, X_0=x_0\}$  should yield a first passage time density, $\rho(t,\theta|x_0)$, that is jointly continuous and $k$-fold differentiable in $t,\theta$ and $x_0$.  Pauwels' conditions are satisfied by the time homogeneous process considered above.   In \cite{Lehmann2002AdvApplProb} Lehmann used an integral equation approach to extend these results, giving conditions guaranteeing H\"{o}lder continuity of the FPT distribution. 

Several authors have obtained expressions for moments of the FPT distribution for the time homogeneous OU process to constant and some moving barriers, for instance \cite{Tuckwell+Wan:1984:JAppPr,Wan+Tuckwell:1982:JTheorNeurobiol}.  Lindner \cite{Lindner:2004:JStatPhys} calculated
moments of the FPT for both exponentially decaying and periodic driving terms similar to the situation described in Equation (\ref{eq:OUP-periodic}) below.  However, knowing the moments of the FPT distribution does not provide a strictly positive lower bound for the tail of the density.  

In 1977 Sato proved that the first passage time density $\rho_c(t)$ for a Wiener process from $B(0)=0$ to a square root boundary $c\sqrt{t+1}$ scales as
\begin{equation}
\rho_c(t)\sim \alpha t^{-p(c)-1}
\end{equation}
where $0<p(c)<1/2$, and it is assumed that $c>0$.  This provides a lower bound for the tail of the density (in the case considered) since 
if $\rho_c(t)/\left(\alpha t^{-p(c)-1}\right)\to 1$ as $t\to\infty$ then for any $\epsilon>0$, there is a $T$ such that for $t>T$ we have $\rho_c(t)>\alpha t^{-p(c)-1}/(1+\epsilon)$.  Sato's result does not provide a lower bound for the suprathreshold OUP, however.  The suprathreshold case corresponds instead to a boundary of the form  $-1+c\sqrt{t}$ after appropriate scaling, because in this case the moving boundary crosses the mean value of the Wiener process, and the expected first passage time is finite.  In the case considered by Sato, the mean first passage time is infinite, and the distance from the mean $\E[B]=0$ to the boundary grows monotonically.  

\subsection{Prospective application to neural synchronization}
\label{sec:discuss2}

Strictly positive lower bounds are also of interest in the case of OU processes with time varying forcing.  The timing of action potentials in nerve cells stimulated by fluctuating current injections are of significant interest in neuroscience \cite{kn:HMTC98,Mainen95,neurocomputing:ThomasEtAl:2003}.  There are many models available for theoretical studies of synchronization of nerve cells, but for models incorporating the effects of noise the LIF model remains among the most tractable, attracting sustained attention \cite{Shimokawa+Pakdaman+Takahata+Tanabe+Sato:2000:BiolCyb,Tateno2002,pla:TatenoJimbo:2000}.  For this case the model dynamics in Equation \ref{eq:dX-OUP} is extended to incorporate the drive through the time varying function $h(t)$, assumed to be bounded, measurable and of zero mean:
\begin{equation}\label{eq:OUP-periodic}
dX=(-\beta X + h(t))\,dt + \sigma\,dW.
\end{equation}
When the forcing function $h$ is periodic with period $T$, one is interested in the succession of first passage times $\{\tau_1,\tau_2,\cdots\}$ assuming the process is instantaneously reset to $x_0$ upon each encounter with the threshold. It can be shown that the sequence of phases of the boundary crossings relative to the periodic drive, $\phi_n = \tau_n\mod T$, then forms a Markov process on the circle $\mathcal{S}^1\equiv[0,1)$, with transition probabilities 
given by a map $\mathbf{K}:f_{k}\to f_{k+1}$ where 
\begin{equation}\label{eq:integral-eqn-K}
f_{k+1}(\phi)=\int_{\psi\in\mathcal{S}^1}K(\phi,\psi)f_k(\psi)\,d\psi.
\end{equation}  
Here $f_k(\phi)$ is the density function for the phase of the $k^{th}$ boundary crossing relative to the periodic drive.  This framework has been elaborated to study stochastic synchronization of systems analogous to the LIF model neuron driven by periodic injected current \cite{TatenoDoiSatoRicciardi1995JStatPhys} as well as LIF model neurons with constant drive and periodically varying $\theta(t)$ \cite{Shimokawa+Pakdaman+Takahata+Tanabe+Sato:2000:BiolCyb,Tateno1998JStatPhys,Tateno2002}.  In both types of models, numerical investigations suggest that regardless of initial conditions the phase distribution converges to a unique probability density on the circle \cite{Shimokawa+Pakdaman+Takahata+Tanabe+Sato:2000:BiolCyb,TatenoDoiSatoRicciardi1995JStatPhys}.  This situation contrasts with that of the \emph{deterministic} periodically forced LIF model neuron, in which a rich collection of resonances and $p:q$ mode locking are known to occur \cite{Bressloff99,siam:Keener+Hoppensteadt+Rinzel:1981,Pakdaman2001PRE,bmb:RescignoEtAl:1970}.
The existence of $p:q$ mode locked solutions ($q$ boundary crossings per $p$ stimulus periods) implies the existence of multiple invariant measures on the circle when $p>1$.  

Tateno \textit{et al.}~conjectured that 
the operator $\mathbf{K}$ for the stochastic periodically forced LIF model neuron has a unique invariant density and that the sequence of probability densities $\{\mathbf{K}^n f_0\}$ is asymptotically stable (\cite{TatenoDoiSatoRicciardi1995JStatPhys}, \S 5-6).  
One approach to proving this conjecture would be to exploit the following result (with $m=1$):

\begin{thm}[\textbf{Corollary 5.7.1} in \cite{LasotaMackey94}; see also \cite{TatenoDoiSatoRicciardi1995JStatPhys}, Equation (38)]
\textit{ Let $(X,\mathcal{A},\mu)$ be a measure space, $K:X\times X\to \mathbb{R}$ a stochastic kernel, and $\mathbf{K}$ the corresponding Markov operator (defined by the integral equation (\ref{eq:integral-eqn-K}) above).  Denote by $K_n$ the kernel corresponding to $\mathbf{K}^n$.  If, for some $m$,
\begin{equation}\label{eq:infimum}
\int_X \inf_y K_m(x,y)\,dx>0
\end{equation}
then $\{P^n\}$ is asympotically stable.}
\end{thm}

Exhibiting a strictly positive lower bound analogous to (\ref{ineq:main-result}) for the periodically forced case would immediately establish the infimum criterion (\ref{eq:infimum}).  The criterion has been assumed to hold under various circumstances (\textit{cf.} Equation (6) in \cite{Doi+Inoue+Kumagai1998JStatPhys}; page 330 of \cite{Shimokawa+Pakdaman+Takahata+Tanabe+Sato:2000:BiolCyb}; and Equation (29) in \cite{Tateno1998JStatPhys}) but to our knowledge an explicit lower bound for the FPT density for a suprathreshold system has not been obtained either for the periodically forced LIF model with constant boundary nor for the time homogeneous LIF model with oscillating boundary.  The extension of Theorem \ref{thm:main} to one or another of these cases remains a topic for future research.  

\section{Acknowledgments}
This work was supported in part by NSF grants DMS-0720142 and DMS-1010434 in the mathematical biology program.  I am grateful for discussion and encouragement to the following: D. Calvetti, M. Denker, R. Galan, K. Kirkpatrick, K. Loparo, E. Somersalo, S. Szarek, and W. Woyczynski.  I thank an anonymous referee for providing valuable feedback.  I am grateful to the Oberlin College Library for research support.

% Place the text of your acknowledgements after the \acks command.
% \acks generates the heading "Acknowledgements".
% If you wish to make only one acknowledgement, use \ack.
% \ack generates the heading "Acknowledgement".

% Reference list
%
% References should be in the following form (or the BibTeX file
% apt.bst should be used):
%
% For a journal:
% Surname, Initial (year). Title of paper. {\em Journal title}
% {\bf Vol,} page--range.
%
% For a book:
% Surname, Initial (year). {\em Book title}. Publisher, Address.
%
% Note the following example of a reference list.

%\begin{thebibliography}{99}
%\footnotesize

\bibliographystyle{apt}
\bibliography{math,reliability,stoch_chem,PJT,Dicty,neuroscience,Respiration}

%\bibitem{ref1}
%{\sc Ball, K. and Chain, H.} (1988). {\em Kurtosis: A Critical Review}, 2nd~edn. John Wiley, New York.

%\bibitem{ref2}
%{\sc Boyd, W.} (1978). Hyperbolic distributions. Doctoral Thesis, University of Boston School of Mathematics.

%\bibitem{ref3}
%{\sc Sichel, H.~S., Kleingeld, W.~J. and Assibey-Bonsu, W.} (1992).  A comparative study of three frequency-distribution models for use in ore valuation. {\em J. S. Afr. Inst. Min. Met.} {\bf 92,} 91--99.

%\end{thebibliography}

\end{document}